\documentclass[12pt]{amsart}
\usepackage{amsmath,amsthm,amsfonts,amssymb}

\newcommand{\ph}{\operatorname{PH}}
\newcommand{\diff}{\operatorname{Diff}}

\newcommand{\nuh}{\operatorname{Nuh}}
\renewcommand{\le}{\operatorname{LE}}
\newcommand{\N}{\mathbb N}
\newcommand{\M}{\mathcal M}
\newcommand{\Z}{\mathbb Z}
\newcommand{\R}{\mathbb R}
\newcommand{\merg}{{\mathcal M}^{erg}}
\newcommand{\eps}{\varepsilon}
\renewcommand{\P}{\mathcal P}
\newcommand{\supp}{\operatorname{supp}}
\newtheorem{theorem}{Theorem}[section]
\newtheorem{conjecture}[theorem]{Conjecture}
\newtheorem{definition}[theorem]{Definition}
\newtheorem{proposition}[theorem]{Proposition}

\newtheorem{question}[theorem]{Question}

\title{Some advances on generic properties of the Oseledets splitting}
\author{Jana Rodriguez Hertz}
\dedicatory{\begin{flushright}Dedicated to the memory of Ricardo Ma\~n\'e, \\ on occasion of the 15th anniversary of his death.\end{flushright}}
\begin{document}
\begin{abstract}
In his foundational paper \cite{manhe1983} , Ma\~n\'e suggested that some aspects of the Oseledets splitting could be improved if one worked under $C^1$-generic conditions. He announced some powerful theorems, and suggested some lines to follow. Here we survey the state of the art and some recent advances in these directions.
\end{abstract}
\maketitle
\section{Introduction}
In his foundational paper \cite{manhe1983} , Ma\~n\'e suggested that some aspects of the Oseledets splitting could be improved if one worked under $C^1$-generic conditions. He announced some powerful theorems, and suggested some lines to follow. Here we survey some recent advances in these directions. \par
Given a $C^1$ diffeomorphism $f\in\diff^1(M)$, a point $x\in M$ and a tangent vector $v\in T_xM\setminus\{0\}$, the {\em Lyapunov exponent} of $x$ associated to $v$ is given by:
\begin{equation}\label{equation Lyapunov exponent}
\lambda(x,v)=\lim_{n\to\pm\infty}\frac{1}{n}\log \|Df^n(x)v\|
\end{equation}
If this limit exists, we can roughly say that $\lambda(x,v)$ measures the asymptotic exponential expansion of the norm of $v$ under the action of the derivative $Df^n(x)$. A point $x$ is called {\em regular} if there exists an invariant splitting $T_xM=E_1(x)\oplus\dots\oplus E_{k(x)}(x)$ such that
\begin{equation}\label{oseledets splitting}
\lambda(x,v_i)=\lambda_i(x)\qquad\mbox{for all}\quad v_i\in E_i(x)\setminus\{0\}\qquad i=1,\dots,k(x)
\end{equation}
The splitting above is unique and is called the {\em Oseledets splitting}. In the sixties, Oseledets proved \cite{oseledec} that for all $C^1$ diffeomorphism $f$, there is a total probability set $R(f)$ of regular points. Namely, for all $f$-invariant measure $\mu$ we have $\mu(R(f))=1$.\par
Let us note that the Oseledets splitting, as well as the Lyapunov exponents, depend only measurably on the base point. This is natural, on the other hand, since Oseledets theorem holds for {\em any} $C^1$ diffeomorphism. In \cite{manhe1983}, Ma\~n\'e suggested that one could improve the regularity of the Oseledets splitting by working under $C^1$ generic conditions. \par
This can be done in different ways. One approach consists in looking at the Oseledets theorem for generic invariant measures for generic diffeomorphisms. This is dealt with next in Subsection \ref{subsection generic} and Section \ref{section.generic}. Another approach, also mentioned in \cite{manhe1983}, consists in studying Oseledets theorem with respect to a measure coming from a symplectic form $\omega$, and considering generic diffeomorphisms preserving this form. We will extend this concept in Subsection \ref{subsection.symplectic} and Section \ref{section.symplectic}.\par
Finally, one can deal with the improvement of Oseledets splitting properties for $C^1$-generic {\em volume preserving} diffeomorphisms. This approach was only suggested in \cite{manhe1983}, though many advances have been made in this line in the recent years. We focus on this point of view in Subsection \ref{subsection.conservative} and Section \ref{section.conservative}.\par

Let us note that in the conservative setting one aims at a $C^1$-generic dichotomy: non-uniform hyperbolicity, global domination and ergodicity vs. zero Lyapunov exponents almost everywhere. This was stated explicitly as a conjecture by Avila and Bochi in \cite{avila_bochi2009}, based on results in \cite{manhe1983,bochi2002} and \cite{jrhertz2010}. Although many partial results have been obtained, this has not been achieved (or disproved) yet. See \ref{subsection.conservative} and Section \ref{section.conservative} and references therein.\par
On the other hand, in the symplectic setting one has non-removable zero Lyapunov exponents in the interior of partially hyperbolic diffeomorphisms, so the above dichotomy is not expectable. However, one can aim at the $C^1$-generic dichotomy: partially hyperbolicity and ergodicity vs. zero Lyapunov exponents almost everywhere (see Subsection \ref{subsection.symplectic}, Section \ref{section.symplectic} and references therein). \par

\subsection{Generic diffeomorphisms} \label{subsection generic}
Let us state in this subsection the lines suggested by Ma\~n\'e in \cite{manhe1983}, with  respect to the first approach. \par
Let us denote by $\M(M)$ the set of probabilities defined in $M$, endowed with the weak topology; and consider $\M(f)$ the subset of $f$-invariant probability measures. This is a compact metric space, and hence it is a Baire space. As usual we shall call a set {\em residual} if it contains a countable intersection of open and dense sets. Recall that in a Baire space residual sets are dense. We shall call a set {\em meager} if its complement is residual. Let us denote by $\merg(f)$ the subset of $\M(f)$ consisting in ergodic measures. This set is a $G_\delta$ subset of $\M(f)$ and therefore it is a Baire space too. We shall say that a property holds for a {\em generic} element in a Baire space if this property holds for a residual subset of the Baire space.\par
Given a measure $\mu\in \M(M)$, we shall denote by $\supp(\mu)$ the support of $\mu$, that is, the set of points such that all of its neighborhoods have $\mu$-positive measure. Recall that this is a compact set.
We have that:
\begin{theorem}\cite{manhe1983}\label{teo.oseledets.generico} For a generic $f\in\diff^1(M)$ and a generic ergodic measure $\mu\in\merg(f)$, the Oseledets splitting extends to a dominated splitting over $\supp(\mu)$.
\end{theorem}
This in particular implies that the Oseledets splitting varies continuously over $\mu$-almost every point in $\supp(\mu)$, and has bounded angles. See precise definition of dominated splitting in Definitions \ref{definition dominated splitting} and \ref{definition dominated splitting muchos}. Observe that when the Oseledets splitting is trivial, we have that the Lyapunov exponents are all zero $\mu$-almost everywhere.\par
We must warn the reader that, as Ma\~n\'e noticed in \cite{manhe.libro}, the set of regular points $R(f)$, over which the Oseledets splitting is defined, is in general meager. A precise statement and proof of this is given by Abdenur, Bonatti and Crovisier in \cite{abdenur_bonatti_crovisier2008}. See Section \ref{section.generic}.\par
However, the support of $\mu$ could be large from the topological point of view, and it could be even the whole manifold $M$. \newline\par
The theorem above is based on the generic denseness of periodic measures among ergodic ones. This result, also stated in \cite{manhe1983}, is based on the Ergodic Closing Lemma \cite{manhe1982} and the Birkhoff Theorem, and it is interesting in itself, since it is the base of many results in this subject. \par
Given a diffeomorphism $f\in \diff^1(M)$ and an $f$-periodic orbit $\gamma$, let us denote by $\mu_\gamma$ the measure equidistributed along $\gamma$, that is, if $m$ is the period of $\gamma$ and $x\in\gamma$, then:
\begin{equation}\label{periodic measure}
\mu_\gamma =\frac{1}{m}\sum_{j=0}^{m-1}\delta_{f^j(x)}
\end{equation}
any such measure will be called {\em periodic measure}. Let us denote by $\P er(f)\subset\merg(f)$ the set of periodic measures.
In \cite{manhe1983}, Ma\~n\'e proved that
\begin{theorem}\cite{manhe1983} \label{teo.aprox.period.manhe}For a generic $f\in \diff^1(M)$, the set of periodic measures $\P er(f)$ is dense in $\merg(f)$.
\end{theorem}
In fact, as a direct consequence of the Ergodic Closing Lemma \cite{manhe1982}, we have that the support of the periodic measures also approach the support of every ergodic measure in the Hausdorff topology. See definitions in Section \ref{section.preliminaries}. This result has been later improved by Abdenur, Bonatti and Crovisier in \cite{abdenur_bonatti_crovisier2008}, see Section \ref{section.generic} for a precise statement.\par
With respect to Theorem \ref{teo.oseledets.generico}, let us observe that there are obviously some limitations when talking about {\em generic} ergodic measures of generic diffeomorphisms. One obvious limitation is, as Ma\~n\'e himself noticed, that the metric entropy with respect to a generic invariant measure, is zero. This implies that, the generic set of $\merg(f)$ for which Theorem \ref{teo.oseledets.generico}, may fail to reflect the true dynamics of $f$. See \cite{abdenur_bonatti_crovisier2008} and Section \ref{section.generic} for more details. \par
Another possible limitation is that, as Avila and Bochi showed in \cite{avila_bochi2006}, $C^1$-generic diffeomorphisms do not admit absolutely continuous invariant measures, so it is not expectable, for the set of diffeomorphisms described in Theorem \ref{teo.oseledets.generico} to have a natural measure. Even in the case it had, it is not expectable that this measure be in the set of generic ergodic measures described in Theorem \ref{teo.oseledets.generico}.
\subsection{Generic symplectic diffeomorphisms}\label{subsection.symplectic} Let consider now a $2n$-dimensional manifold $M$, and let $\omega$ be a symplectic form on $M$, that is, a non-degenerate closed 2-form. If we consider $\mu=\omega\wedge\stackrel{n}{\dots}\wedge\omega$, we obtain a volume form on $M$. We shall say a diffeomorphism $f$ is {\em symplectic} if $f$ preserves this form $\omega$. We denote by $\diff^1_\omega(M)$ the set of such diffeomorphisms.\par
In \cite{manhe1983}, Ma\~n\'e said he would address generic properties of Oseledets splitting for symplectic diffeomorphisms instead of volume preserving diffeomorphisms. He adduced the following two technical reasons: first, partial hyperbolicity follows from domination in the symplectic setting (see for instance \cite{bochi_viana2004} for a proof of this fact), which was what he claimed to have proven. Second, in the proof, he needed to approximate a $C^1$-volume preserving diffeomorphism by a $C^2$-volume preserving diffeomorphism, a tool that was not available at that time, but was possible in the symplectic setting \cite{zehnder1977}.\par
The second obstacle was tackled recently by Avila \cite{avila2008}. The first obstacle has been tackled in dimension 3 by the author in \cite{jrhertz2010}, where it is proved that generically the Oseledets splitting is globally dominated. This further implies generic ergodicity in the zone where the Oseledets splitting is not trivial (see Subsection \ref{subsection.conservative} and Section \ref{section.conservative}). However, this seems difficult yet to obtain in higher dimensions in the conservative setting, we discuss this topic in the next subsection.\newline\par
Given $f\in\diff^1_\omega(M)$ and $f\in R(f)$, let us define the following sub-bundles:
$$E^+(x)=\oplus\{E_i(x):\lambda_i(x)>0\}\qquad  E^-(x)=\oplus\{E_i(x):\lambda_i(x)<0\}$$
and $$E^0(x)=\oplus\{E_i(x):\lambda_i(x)=0\}.$$
For further use, we denote, following Bochi \cite{bochi2010}, the {\em zipped Oseledets splitting}:
$$T_xM=E^+(x)\oplus E^0(x)\oplus E^-(x)$$
Let us note that for a symplectic $f$, the Lyapunov exponents have the following symmetry: if $\lambda_i(x)$ is a Lyapunov exponent then $-\lambda_i(x)$ is also a Lyapunov exponent with the same multiplicity as $\lambda_i(x)$. Therefore, we have that $\dim E^+(x)=\dim E^-(x)$, and $\dim E^0(x)$ is even. Let us divide the set of regular points $R(f)$ into the following disjoint regions. The {\em elliptic region} is the set $Z(f)=\{x\in R(f):T_xM=E^0(x)\}$, that is the invariant set of points having all its Lyapunov exponents zero. The {\em hyperbolic region} is the set $\nuh(f)=\{x\in R(f):T_xM= E^-(x)\oplus E^+(x)\}$, that is, the Pesin region of $f$ defined below in equation (\ref{equation.pesin.region}). And finally, the {\em partially hyperbolic region} $PH(f)=R(f)\setminus(PR(f)\cup Z(f))$, the set of points having a non-trivial center bundle, which is not all $T_xM$.\par
In \cite{manhe1983}, Ma\~n\'e claimed to have proven the following, which was only proven very recently by Bochi:
\begin{theorem}[\cite{bochi2010}] \label{teo.manhe.symplectic}For a generic $f\in \diff^1_\omega(M)$ only one of the following option holds:
\begin{enumerate}
\item $\mu$-almost every $x\in M$, all Lyapunov exponents are zero.
\item $f$ is Anosov
\item $M=Z\cup PH$ ($\mod 0$) with $\mu(PH)>0$. Then we have $PH=\bigcup_{n>0}D_n$ ($\mod 0$) where $D_n$ are $f$-invariant sets over which $f$ is partially hyperbolic. The Oseledets splitting of $f$ extends to a dominated splitting over each $\overline{D_n}$.
\end{enumerate}
\end{theorem}
Let us note that this has been achieved by Bochi with techniques that are very different from those proposed by Ma\~n\'e. Indeed, he used new perturbation techniques that involved random walks.
Also, in case one restricts to the open set of partially hyperbolic symplectomorphisms, Bochi can improve Theorem \ref{teo.manhe.symplectic} by proving that generically, all center Lyapunov exponents are zero \cite{bochi2010}. See more in Section \ref{section.symplectic}.\par
Let us note that in item 3, one could have in principle for a symplectomorphism that there is a positive measure set on which all Lyapunov exponents vanish, coexisting with a positive measure set where $f$ is partially hyperbolic. We believe this is not the generic case.
\begin{conjecture}\label{conjecture.symplectic}
For a generic $f\in \diff^1_\omega(M)$ either all Lyapunov exponents vanish almost everywhere, or else the Oseledets splitting extends to a globally dominated (partially hyperbolic) splitting and in this case $f$ is ergodic.
\end{conjecture}
Obviously, the last case includes the case where the symplectomorphism is Anosov. Note that generically, partially hyperbolic symplectomorphisms are ergodic, as has been proved by Avila, Bochi and Wilkinson in \cite{avila_bochi_wilkinson2009}. See more in Section \ref{section.symplectic}.\par
As a corollary of Theorem \ref{teo.manhe.symplectic}, Ma\~n\'e announced the following celebrated result:
\begin{theorem}[Bochi \cite{bochi2002}] For a generic area-preserving diffeomorphism $f$ of a surface, either $f$ is Anosov, or else all Lyapunov exponents vanish almost everywhere.
\end{theorem}
He left some outline of the proof which was published after his death in \cite{manhe.montevideo}. This was only a sketch, and substantial work had to be done by Bochi to finish this proof, which appeared in 2002 \cite{bochi2002}. We further develop this topic in Section \ref{section.symplectic}.
\subsection{Generic conservative diffeomorphisms}\label{subsection.conservative} Finally, let us consider the set of diffeomorphisms preserving a smooth volume form $m$ on $M$, which we denote by $\diff^1_m(M)$. As we stated above, this case was not addressed by Ma\~n\'e essentially due to two technical reasons: first, at that time it was not known whether one could approximate a diffeomorphism in $\diff^1_m(M)$ by a smooth volume preserving diffeomorphism. This question has been solved by Avila \cite{avila2008}, who proved that smooth conservative diffeomorphisms are dense in $\diff^1_m(M)$. The second obstacle is even more delicate. In the symplectic setting, dominance implies partial hyperbolicity. This is not true in the conservative setting, not even generically. See more in Section \ref{section.conservative}. Moreover, unlike in the symplectic setting, it is not known if dominance implies generically some ergodicity. \par
This does not prevent us from having a conjecture as a leading guide in trying to establish the overall generic behavior of the Oseledets splitting. The following, stated by Avila and Bochi in 2009 seems to be a good starting point:
\begin{conjecture}\cite{avila_bochi2009} \label{conjecture.avila.bochi}For a generic conservative diffeomorphism $f$ in $\diff^1_m(M)$, one of the following alternatives holds. Either
\begin{enumerate}
\item all Lyapunov exponents of $f$ are zero almost everywhere, or
\item $f$ is non-uniformly hyperbolic and ergodic, and the zipped Oseledets splitting $T_xM=E^+(x)\oplus E^-(x)$, defined in Subsection \ref{subsection.symplectic}, is globally dominated.
\end{enumerate}
\end{conjecture}
The conjecture above has been inspired by Theorem \ref{teo.symplectic.bochi} obtained by Bochi for surfaces, and by the following result, obtained by the author, for 3-manifolds, where it is proved that:
\begin{theorem}\cite{jrhertz2010} For a generic conservative diffeomorphism $f$ in $\diff^1_m(M)$, where $M$ is a 3-dimensional manifold, the Oseledets splitting extends to a dominated splitting over the whole manifold. In particular, either:
\begin{enumerate}
\item all Lyapunov exponents of $f$ are zero almost everywhere, or
\item $f$ is non-uniformly hyperbolic and ergodic.
\end{enumerate}
\end{theorem}
Note that in this case, not only the zipped Oseledets splitting is globally dominated as in item 2 of Conjecture \ref{conjecture.avila.bochi}, but also the finest Oseledets splitting. See more details in Section \ref{section.conservative}. \par
So far, in greater dimensions, the best approximation to Conjecture \ref{conjecture.avila.bochi} is the following result by Avila and Bochi:
\begin{theorem}\cite{avila_bochi2009}\label{teo.avila.bochi} For a generic conservative diffeomorphism $f$ in $\diff^1_m(M)$, only one of the following alternatives holds. Either:
\begin{enumerate}
\item the Pesin region of $f$ has null measure: $m(\nuh(f))=0$; that is, almost everywhere, there is a zero Lyapunov exponent for $f$, or
\item the orbit of almost every orbit in $\nuh(f)$ is dense, and $f$ restricted to $\nuh(f)$ is ergodic. Moreover, the zipped Oseledets splitting $T_xM=E^+(x)\oplus E^-(x)$ extends to a globally dominated splitting.
\end{enumerate}
\end{theorem}
Let us note that, in fact, the Oseledets splitting over the Pesin region $\nuh(f)$ extends to a dominated splitting over the manifold. However, there could be other ergodic components with positive measure, on which the Oseledets splitting is finer (it cannot be coarser). Another thing to mention is that possibility (1) does indeed happen in every manifold: Grin provides in each manifold $M$ an open set ${\mathcal U}\subset\diff^1_m(M)$, such that the generic diffeomorphism in ${\mathcal U}$ has a vanishing Lyapunov exponent almost everywhere. Moreover, in 3-manifolds, this open set can be chosen so that the generic diffeomorphism in ${\mathcal U}$ all Lyapunov exponents vanish almost everywhere \cite{grin2010}.\par
One of the main obstacles to tackle in Theorem \ref{teo.avila.bochi} is the possibility of the coexistence of a positive measure set $Z$ of points with all vanishing Lyapunov exponents with a positive measure set $P$ of points (in some $D_n$ defined below) such that the Oseledets splitting has a non-trivial bundle $E^0(x)$ corresponding to zero Lyapunov exponents. In this case, a global dominated splitting would not exist, and $f$ would be not ergodic.\par
We should also take into account a previous result by Bochi and Viana, which states the following:
\begin{theorem}\cite{bochi_viana2005} \label{teo.bochi.viana} For a generic conservative diffeomorphism $f$ in $\diff^1_m(M)$, $m$-almost every $x$ in $M$, either all Lyapunov exponents of $f$ are zero, or else the Oseledets splitting over the orbit of $x$ is dominated.
\end{theorem}
As a consequence of Theorem \ref{teo.bochi.viana}, the manifold can be written modulo zero as the countable union of a set $Z$ of points where all the Lyapunov exponents of $f$ are zero, and sets $D_n$ where the Oseledets splitting over the orbit of $x$ is $n$-dominated. \par
There are some hypotheses under which Conjecture \ref{conjecture.avila.bochi} could be established, in fact, a positive answer to the following question, attributed to Katok by Bochi and Viana in \cite{bochi_viana2005}, would imply Conjecture \ref{conjecture.avila.bochi}, as we explain in Section \ref{section.conservative}:
\begin{question}\label{question.katok}
Is ergodicity a generic property in $\diff^1_m(M)$?
\end{question}
However, it is likely that establishing this question could be harder than answering Conjecture \ref{conjecture.avila.bochi}. Nevertheless, there are intermediate hypotheses that could give a better global description of the generic behavior of the Oseledets splitting. This is the case, for instance, of {\em weak ergodicity}. A conservative diffeomorphism is weakly ergodic when the only compact invariant set with positive measure is the whole manifold, see more details in Definition \ref{def.weak.ergodicity}. The following question, weaker than Question \ref{question.katok}, is also unsolved, see also \cite{bonatti_crovisier2004}:
\begin{question}
Is weak ergodicity a generic property in $\diff^1_m(M)$?
\end{question}
Unlike in the previous context of ergodicity, Conjecture \ref{conjecture.avila.bochi} does not follow from weak ergodicity, at least not in an immediate way. However, weak ergodicity would improve our working conditions, for instance, we do have the following generic dichotomy: either all Lyapunov exponents vanish almost everywhere, or else the finest Oseledets splitting over a positive measure set, extends to a global dominated splitting. We further discuss this topic in Section \ref{section.conservative}.
\section{Preliminaries}\label{section.preliminaries}
Let $M$ be a compact Riemannian manifold. Let us denote by $\diff^1(M)$ the set of $C^1$ diffeomorphisms, endowed with the $C^1$ topology.\par
As we have stated in the introduction, given a diffeomorphism $f\in\diff^1(M)$,  for a total probability set of points there exists a splitting, the {\em Oseledets splitting}, $T_xM=E_1(x)\oplus\dots\oplus E_{k(x)}(x)$ and numbers $\hat\lambda_1(x)>\dots>\lambda_{k(x)}(x)$, called the {\em Lyapunov exponents} such that for every non-zero vector $v\in E_i(x)$, we have
$$\lambda(x,v)=\hat\lambda_i(x),$$
where $\lambda(x,v)$, defined in Equation (\ref{equation Lyapunov exponent}), is the exponential growth of the norm of $Df^n(x)$ along the direction of $v$. The dimension of each $E_i(x)$ is called the {\em multiplicity of} $\hat\lambda_i(x)$. Counting each Lyapunov exponent with its multiplicity, we obtain, for an $n$-manifold, that
$$\lambda_1(x)\geq\lambda_2(x)\geq\dots\geq\lambda_n(x).$$

One of the generic properties we would like to establish for the Oseledets splitting of a $C^1$ generic diffeomorphism its domination. Let us recall this concept:
\begin{definition} \label{definition dominated splitting} Given an $f$-invariant set $\Lambda$, and two invariant sub-bundles of $T_\Lambda M$, $E$ and $F$, such that $T_\Lambda M=E_\Lambda\oplus F_\Lambda$, we call this splitting an {\em $l$-dominated splitting} if for all $x\in\Lambda$ and all unit vectors $v_E\in E_x$ and $v_F\in F_x$ we have:
\begin{equation}\label{equation dominated splitting}
\|Df^l(x)v_F\|\leq \frac12\|Df^l(x)v_E\|
\end{equation}
we denote $E_\Lambda\succ_l F_\Lambda$ or $F_\Lambda\prec E_\Lambda$
\end{definition}
Note that we do not require $\Lambda$ to be compact. In particular, we shall denote $E_x\prec_l F_x$ when the inequality (\ref{equation dominated splitting}) is satisfied for the orbit of $x$.
A splitting will be called {\em dominated} if there exists $l\in\N$ such that it is $l$-dominated. When a splitting is dominated, a vector not in $E$ or $F$ will converge to $F$ under forward iterates and to $E$ under backward iterates.
\begin{definition}\label{definition dominated splitting muchos}
If we have invariant sub-bundles $T_\Lambda M=E^1_\Lambda\oplus\dots\oplus E^k_\Lambda$ over an invariant (not necessarily compact) set $\Lambda$, we will say that the splitting is {\em $l$-dominated} if $$E^1_\Lambda\oplus\dots\oplus E^i_\Lambda\succ_l E^{i+1}_\Lambda\oplus\dots E^k\Lambda\qquad\forall i=1,\dots,k-1$$
\end{definition}
As before we will say that the splitting is dominated if it is $l$-dominated for some $l$. A particular case of a dominated splitting is the following:
\begin{definition}\label{definition partially hyperbolic} We shall say that the invariant splitting
$$T_\Lambda=E_\Lambda^s\oplus E_\Lambda^c\oplus E_\Lambda^u$$
over the invariant set $\Lambda$ is {\em $l$-partially hyperbolic} if it is $l$-dominated and $Df$ is contracting on $E_\Lambda^s$ and $Df^{-1}$ is contracting on $E_\Lambda^u$.
\end{definition}
Analogously, we shall say that the  splitting is {\em partially hyperbolic} if it is $l$-partially hyperbolic for some $l\in \N$.\newline\par
We shall be particularly interested in measures not having zero Lyapunov exponents. Namely,
\begin{definition}\label{definition.hyperbolic.measure} Given $f\in\diff^1(M)$, an ergodic measure $\mu\in\merg(f)$ will be called a {\em hyperbolic measure} if all Lyapunov exponents are different from zero $\mu$-almost everywhere.
\end{definition}
In the conservative setting, we shall identify the region where exponents are different from zero as the {\em Pesin region}. That is, the Pesin region of $f\in\diff^1_m(M)$ is the set $\nuh(f)$ such that:
\begin{equation}\label{equation.pesin.region}
\nuh(f)=\{x\in R(f):\lambda(x,v)\ne0\quad\forall v\in T_xM\setminus\{0\}\}
\end{equation}
We shall say that a conservative diffeomorphism $f$ is {\em non-uniformly hyperbolic} if $\nuh(f)$ equals the whole manifold modulo a zero measure set.

\section{$C^1$ generic behavior of the Oseledets splitting in $\diff^1(M)$}\label{section.generic}
In this section we shall consider the advances that have been made towards describing the generic behavior of the Oseledets splitting in the dissipative setting. A good account of much of this progress can be found in the article by Abdenur, Bonatti and Crovisier \cite{abdenur_bonatti_crovisier2008}.\par
Indeed, Theorem \ref{teo.oseledets.generico} has been improved with the following information:
\begin{theorem}\cite{abdenur_bonatti_crovisier2008} For a generic diffeomorphism $f\in \diff^1(M)$ and a generic measure $\mu \in \merg(f)$, we have:
\begin{enumerate}
\item $\mu$ has zero entropy: $h_\mu(f)=0$.
\item $\mu$ is a hyperbolic measure.
\end{enumerate}
\end{theorem}
A sketch of the proof of item 1 was already given in \cite{manhe1983}. In \cite{abdenur_bonatti_crovisier2008} this result is re-obtained as a corollary of a more general statement. This result, as in the case of Theorem \ref{teo.oseledets.generico} is based on the following improvement of Theorem \ref{teo.aprox.period.manhe}:
\begin{theorem}\cite{abdenur_bonatti_crovisier2008} For a generic diffeomorphism $f\in \diff^1(M)$, and {\em any} ergodic measure $\mu\in\merg(f)$, there exists a sequence of hyperbolic periodic measures $\mu_n\subset\P er(f)$ such that:
\begin{enumerate}
\item $\mu_n\to\mu$ in the weak topology
\item $\supp(\mu_n)\to\supp(\mu)$ in the Hausdorff topology
\item $\lambda_i(\mu_n)\to\lambda_i(\mu)$ for all $i=1,\dots,n$, where $\lambda_i(\mu_n)$ and $\lambda_i(\mu)$ are the Lyapunov exponents of, respectively, $\mu_n$ and $\mu$, considered with their multiplicity.
\end{enumerate}
\end{theorem}
These results can be improved if one works under additional hypotheses. We shall consider transitivity and isolation.
\subsection{Properties of measures supported on an isolated transitive set} Let $\Lambda$ be an $f$-invariant set for a diffeomorphism $f\in\diff^1(M)$. We shall say that $\Lambda$ is {\em transitive} if it contains a dense orbit, and we shall say that it is {\em isolated} if it has an {\em isolating neighborhood} $V\subset M$ such that
$$\bigcap_{n\in\Z}f^n(V)=\Lambda.$$
For further use, given {\em any} $f$-invariant set $\Lambda$, let us denote by $\M_f(\Lambda)$ the set of $f$-invariant measures supported in $\Lambda$, denote also by $\merg_f(\Lambda)$ the set $\merg(f)\cap\M_f(\Lambda)$ and finally, let us call $\P_f(\Lambda)$ the set of periodic measures supported in $\Lambda$, that is $\P er(f)\cap\M_f(\Lambda)$.\par
In many ways, the dynamic behavior of isolated transitive sets resemble that of basic sets for Axiom A diffeomorphisms. This is the particular case of invariant measures. Let us recall that in the seventies, Sigmund proved the following for the space of invariant measures of a basic set of an Axiom A diffeomorphism:
\begin{theorem}\cite{sigmund1970}\label{theorem.sigmund} Let $\Lambda$ be a basic set of an Axiom A diffeomorphism $f$. Denote by $\M=\M_f(\Lambda)$ the set of invariant probabilities supported in $\Lambda$. Then we have:
\begin{enumerate}
\item Every $\mu\in\M$ is approximated by hyperbolic periodic measures $\mu_n\in\P_f(\Lambda)$.
\item The generic measure $\mu\in\M$ is ergodic.
\item The generic measure $\mu\in\M$ has entropy zero: $h_\mu(f)=0$.
\end{enumerate}
\end{theorem}
As noted by Sigmund, item 2 follows from item 1, and the fact that ergodic measures are a $G_\delta$ set. Item 1, in turn, follows essentially from Bowen's specification property valid for basic sets.
In \cite{abdenur_bonatti_crovisier2008}, Abdenur, Bonatti and Crovisier obtained the following generalization of Theorem \ref{theorem.sigmund} and Theorem \ref{teo.oseledets.generico}:
\begin{theorem}\cite{abdenur_bonatti_crovisier2008} Let $\Lambda$ be an isolated transitive set of a $C^1$-{\em generic} diffeomorphism $f\in\diff^1(M)$. Then we have:
\begin{enumerate}
\item Every $\mu\in\M_f(\Lambda)$ is approximated by hyperbolic periodic measures $\mu_n\in\P_f(\Lambda)$.
\item The generic measure $\mu\in\M_f(\Lambda)$ is ergodic and satisfies $\supp(\mu)=\Lambda$.
\item The generic measure $\mu\in\M_f(\Lambda)$ has entropy zero: $h_\mu(f)=0$.
\item The generic measure $\mu\in\M_f(\Lambda)$ is  hyperbolic.
\item For a generic measure $\mu\in \M_f(\Lambda)$ its Oseledets splitting can be extended to a dominated splitting over $\Lambda$.
\end{enumerate}
\end{theorem}
Here item 1 is not deduced from the specification property but from a weaker feature of periodic points in an isolated transitive set of a generic diffeomorphism called the {\em barycenter property}. Indeed, for a residual set of diffeomorphisms, the set $\P er_f(\Lambda)$ of periodic points of $\Lambda$ satisfies the following: given $p$ and $q$ in $\P er_f(\Lambda)$, and $\eps>0$ there exists an integer $N>0$ such that for any pair of positive integers $n_1$ and $n_2$ there is a periodic point $z\in\P er_f(\Lambda)$ such that
$$d(f^k(z),f^k(p))<\eps\qquad\mbox{for } k=1,\dots,n_1$$
and
$$d(f^k(z),f^k(q))<\eps\qquad\mbox{for } k=n_1+N,\dots, n_1+N+n_2.$$
\subsection{Abundance of non-regular points} Finally, let us remark, as in the introduction, that in general, the set of regular points is meager. Let us say that a point $x\in M$ is {\em regular} for $f\in \diff^1(M)$ if for every continuous function $\varphi:M\to \R$, the following limit exists. A complete proof of this can be found in \cite{abdenur_bonatti_crovisier2008}:
$$\tilde\varphi(x)=\lim_{|n|\to\infty}\frac{1}{n}\sum_{k=0}^{n-1}\varphi\circ f^k(x)$$
As mentioned in \cite{manhe.libro} without proof, regular points are in general few from the topological point of view:
\begin{theorem}\cite{abdenur_bonatti_crovisier2008} For a generic diffeomorphism $f\in\diff^1(M)$, the set of regular points is meager.
\end{theorem}
\section{$C^1$ generic behavior of the Oseledets splitting in $\diff^1_\omega(M)$}\label{section.symplectic}
In this section we review some advances and tools available for generic symplectic diffeomorphisms. There are many authors that have worked in this area, hence we shall only choose some results that we think that are relevant to the line of the generic behavior of the Oseledets splitting in $\diff^1_\omega(M)$.\par
For the rest of this section let us consider $M$ a symplectic manifold and $\omega$ a symplectic form. The volume form coming from $\omega$ will be denoted by $\mu$. As we have mentioned in Subsection \ref{subsection.symplectic}, it was only recently that Theorem \ref{teo.manhe.symplectic} was proved by Bochi. In other words, we have the following:
\begin{theorem}\cite{bochi2010} \label{teo.symplectic.bochi}For a generic diffeomorphism $f\in\diff^1_\omega(M)$, only one of the following options holds:
\begin{enumerate}
\item $f$ is Anosov, or
\item $\mu$-almost every $x$ either all Lyapunov exponents are zero, or the Oseledets splitting is partially hyperbolic over the orbit of $x$.
\end{enumerate}
\end{theorem}
We remark once again that the techniques involved in the theorem above are somehow different from those in the rest of this review. The general strategy, as we shall briefly explain below is as in Bochi and Viana's work \cite{bochi_viana2005}, but the difference the perturbation method in one of the cases. Let us explain a little more, without pretending to get into too much detail. \par
Given $f\in\diff^1_\omega(M)$ and a regular point $x\in M$, consider the Lyapunov exponents counted with multiplicity:
$$\lambda_1(f,x)\geq\dots\lambda_{2n}(f,x)$$
Just as we will do in the conservative case (Section \ref{section.conservative}) let us consider, for each $p=1,\dots,n$ the {\em integrated $p$-exponent} of $f$:
\begin{equation}
\label{equation.integrated.lyapunov.exponent}
\le_p(f)=\int_M(\lambda_1(f,x)+\dots+\lambda_p(f,x))\,d\mu
\end{equation}
In the symplectic setting, it suffices to consider only the first $n$ (integrated) exponents, since the other ones are symmetric, as we have mentioned in Subsection \ref{subsection.symplectic}. Again, just as in the conservative setting, we have that the map $\le_p:\diff^1_\omega(M)\to \R$ is upper-semicontinuous; hence, its points of continuity form a residual subset of $\diff^1_\omega(M)$. The idea of the proof is, as in \cite{bochi_viana2005}, to show that:
\begin{theorem}\cite{bochi2010}\label{teo.symplectic.bochi.2}
If $f\in\diff^1_\omega(M)$ is a point of continuity of $\le_p$ for $p=1,\dots,n$, then $\mu$-almost every $x$ in $M$, the Oseledets splitting is either trivial or dominated along the orbit of $x$.
\end{theorem}
Assume there is a point at which the Oseledets splitting is not trivial nor dominated, that is, the Oseledets splitting admits the invariant decomposition $T_x=E\oplus F$, but $E$ does not dominate $F$. Let $p=\dim E$. There is some iterate $m$ such that $Df^m(x)|E_x$ is ``mixable" with $Df^m(x)|F_x$. Then the idea is to produce a perturbation $g$ supported in the tower $U\sqcup\dots\sqcup f^{m-1}U$, where $U$ is a small neighborhood of $x$, such that $Dg^m(y)$ sends $E(y)$ into $F(y)$ for ``many" points $y$ in $U$. Many here means a large amount in the sense of measure. \par
This perturbation is performed all over the manifold, so that the expansion rates of $E$ and $F$ get mixed, and this causes $LE_p$ fall down abruptly, so a discontinuity point is created. Up to here this is also the same general strategy followed in \cite{bochi_viana2005} for the conservative case. \par
The novelty here is that in one of the cases, the perturbation methods in \cite{bochi_viana2005} do not apply to the symplectic setting. In this case, the perturbation is supported in a domain $f(D)=g(D)$. The domain is then mostly covered with small disjoint boxes, in which one assumes that the direction of $Dg(x)v_E$ is constant, where $v_E$ is a vector in $E$. One wants to perturb the angles so that $v_E$ falls in $F$ for ``many" $y$ in $D$, from the point of view of measure. The idea is to look at the angles as {\em random variables}, which are independent and identically distributed. He produces a construction by successively subdividing the boxes $D_i$ that gives a {\em random walk} on the real line. This random walk eventually gets close to the angle $\pm\frac{\pi}2$ (the angle he needs to get an $E$-vector into $F$ in that particular case). When this happens, he produces a further perturbation so that the angle becomes $\pm\frac{\pi}2$, and then this orbit is no longer perturbed. That is, he produces a random walk with absorbing barriers at $\pm\frac{\pi}2$. Eventually, for the majority of orbits the vector $v_E$ will fall in $F$, and the procedure follows as in the previous case.\newline\par
In short, what Theorem \ref{teo.symplectic.bochi.2} says is that we can decompose the manifold $M$, modulo a zero-measure set, into a countable collection of sets $Z\cup \bigcup_{n\geq 1} D_n$, where $Z$ is the set where all Lyapunov exponents vanish, and $D_n$ is the set of point where the Oseledets splitting is $n$-dominated. One would like to obtain a generic dichotomy such as the one stated in Conjecture \ref{conjecture.symplectic}, that is generically, either $m(Z)=1$ or else, $f$ is ergodic and $m(D_n)=1$ for some $n$. In that case we have that $f$ is partially hyperbolic. \par
There are some situations where this conjecture can be answered. For instance, in case $f$ is ergodic. Indeed, all these sets are invariant and, being countable, one of them has to have positive measure. In this case, we would obtain either all Lyapunov exponents are zero almost everywhere, or else $f$ is partially hyperbolic. \par
Other situation is when a generic $f$ has some $D_n$ with non-empty interior. Indeed, by a result of Arnaud, Bonatti and Crovisier \cite{arnaud_bonatti_crovisier2005}, the generic symplectomorphism is transitive; hence, the presence of a $D_n$ with non-empty interior will imply the existence of a global $n$-dominated splitting, which in the symplectic setting implies global partial hyperbolicity. \par
An intermediate situation is when a generic $f$ is {\em weakly ergodic}. Weak ergodicity is stronger than transitivity and weaker than ergodicity.
\begin{definition}\label{def.weak.ergodicity}
A volume preserving diffeomorphism is {\em weakly ergodic} if almost every orbit is dense. Equivalently, if any invariant set with positive measure is dense. It is also equivalent to the fact that every non-empty open invariant set has total measure.
\end{definition}
In the case of a generic weakly ergodic symplectomorphism, if it exists, either $m(Z)=1$ or else one of the $D_n$ is dense, hence we have partial hyperbolicity, and, as we shall see below, ergodicity.\newline\par
Now, as stated above, in case that the generic symplectomorphism is {\em partially hyperbolic}, we can say more. Note that the subset $\ph^1_\omega(M)$ of partially hyperbolic diffeomorphisms in $\diff^1_\omega(M)$, is an open set. Indeed, in the set of volume preserving partially hyperbolic diffeomorphisms, the hypothesis of accessibility together with some smoothness implies weak ergodicity. A partially hyperbolic diffeomorphism has the {\em accessibility property} if the only set that is simultaneously saturated by stable leaves and unstable leaves is the whole manifold. We have the following, proved by Brin in the seventies:
\begin{theorem}\cite{brin1975}\label{teo.brin.weak.ergodicity} Any volume preserving diffeomorphism in $\diff^{1+\alpha}_m(M)$ having the accessibility property is weakly ergodic.
\end{theorem}
Note that there is no symplectic assumptions here. But as Dolgopyat and Wilkinson showed, the set of partially hyperbolic diffeomorphisms having the accessibility property are abundant both in the conservative and in the symplectic setting:
\begin{theorem}\label{teo.dolgopyat.wilkinson}\cite{dolgopyat_wilkinson2003} The set of partially hyperbolic diffeomorphisms, either conservative or symplectic contains an open and dense set of diffeomorphisms having the accessibility property.
\end{theorem}
On the other hand, as we can see for instance in \cite{bochi2010}, weak ergodicity forms a $G_\delta$ set of diffeomorphisms, either in $\diff^1_m(M)$ or in $\diff^1_\omega(M)$. In fact, it is not difficult to check this, it suffices to see the set of weakly ergodic diffeomorphisms as the countable intersection of the open sets. See also \cite{bochi2010}:
\begin{proposition}\label{proposition.jana.weak.ergodicity.gdelta}
Weakly ergodic diffeomorphisms form a $G_\delta$ set of $\diff^1_\omega(M)$ (and $\diff^1_m(M)$)
\end{proposition}
\begin{proof}
For each open set $V$, let us denote by $V^f$ the set $\bigcup_{n\in\Z}f^n(V)$. Note that the assignment $f\mapsto m(V^f)$ is lower-semicontinuous, since $m(V^f)$ is the supremum of the continuous functions $f\mapsto m(\bigcup_{|n|\leq N}f^n(V))$. This implies that the sets of diffeomorphisms ${\mathcal O}_{V,n}$ formed by all the diffeomorphisms for which $m(V^f)>1-\frac1n$ are open for each fixed open set $V$ and positive integer $n$. Now it is easy to see that the set of weakly ergodic diffeomorphisms coincides with the countable intersection of ${\mathcal O}_{V,n}$, where $V$ varies over a countable base of open sets, and $n$ over the positive integers. 
\end{proof}
Note that this, together with Theorems \ref{teo.brin.weak.ergodicity} and \ref{teo.dolgopyat.wilkinson}, implies that:
\begin{proposition}\label{proposition.weak.ergodicity.generic}
Weak ergodicity is generic among conservative and symplectic partially hyperbolic diffeomorphisms.
\end{proposition}
This fact was used by Bochi to prove that:
\begin{theorem}\cite{bochi2010} For a generic symplectomorphism in $\ph^1_\omega(M)$, there is a partially hyperbolic splitting $TM=E^u\oplus E^c\oplus E^s$, such that all Lyapunov exponents in the center bundle are zero almost everywhere.
\end{theorem}
That is, the Oseledets splitting $E^+\oplus E^0\oplus E^-$ defined in Subsection \ref{subsection.symplectic} extends globally to a dominated (partially hyperbolic) splitting. In particular, this implies that generically partially hyperbolic symplectomorphisms satisfy a property called {\em non-uniform center bunching} condition. This property has been shown by Avila, Bochi and Wilkinson to be enough to obtain ergodicity in this context:
\begin{theorem}\cite{avila_bochi_wilkinson2009} A generic partially hyperbolic symplectomorphism in $\ph^1_\omega(M)$ is ergodic.
\end{theorem}
\par
Finally, let us mention another result in the direction of Conjecture \ref{conjecture.symplectic}, obtained by Saghin and Xia:
\begin{theorem}\cite{saghin_xia2006} For a residual set of diffeomorphisms $f$ in $\diff^1_\omega(M)$, either one of the following holds:
\begin{enumerate}
\item $f$ is partially hyperbolic
\item the set of elliptic $f$-periodic points is dense in $M$.
\end{enumerate}
\end{theorem}
\section{$C^1$ generic behavior of the Oseledets splitting in $\diff^1_m(M)$}\label{section.conservative}
Finally, let us study the generic behavior of the Oseledets splitting in $\diff^1_m(M)$, the set diffeomorphisms preserving a smooth volume form $m$. As we have mentioned in Subsection \ref{subsection.conservative}, our leading goal in this setting is establishing Conjecture \ref{conjecture.avila.bochi}, proposing that generically, only one of the following alternatives hold:\newline
Either
\begin{enumerate}
\item all Lyapunov exponents vanish almost everywhere, or
\item $f$ is non-uniformly hyperbolic and ergodic, and the Oseledets splitting extends to a globally dominated splitting. 
\end{enumerate}
This Conjecture has already been proved in dimension 2 \cite{bochi2002} and dimension 3 \cite{jrhertz2010}. Even though setting these results has not been a trivial task, there are some particularities that helped in their proofs. Let us discuss them to see the new difficulties that could arise in higher dimensions:\par
In the conservative setting, for any dimension, the Lyapunov exponents $\lambda_1(x)\geq\dots\geq\lambda_n(x)$ satisfy:
$$\lambda_1(x)+\dots+\lambda_n(x)=0\qquad m-\mbox{almost every }x$$ 
Hence, in dimension 2, the presence of a non-vanishing Lyapunov exponent in a set of positive measure, readily implies the existence of a non-trivial Pesin region $\nuh(f)$. If, furthermore, one can establish some kind of dominance over $\nuh(f)$, (the most difficult part in this case) one obtains that its closure $\overline{\nuh(f)}$ is a hyperbolic set with positive measure. But $C^1$-generically, a hyperbolic set has either zero measure or it is the whole manifold. We include below a new proof of this fact using a criterion in \cite{rhrhtu2009}, and techniques in \cite{jrhertz2010}. Note that the result holds for manifolds of any dimension:
\begin{proposition}\label{proposition.hyperbolic.sets}
For a generic conservative or symplectic diffeomorphism $f$ in $\diff^1_m(M)$ or $\diff^1_\omega(M)$, $M$ with any dimension, only one of the following holds. Either
\begin{enumerate}
\item $f$ is Anosov, or
\item all hyperbolic sets have zero measure.
\end{enumerate}
\end{proposition}
Let us observe that in \cite{bochi_viana2004}, the authors prove that for every $C^{1+\alpha}$ non-Anosov conservative diffeomorphism, all hyperbolic sets (not only the isolated ones) have measure zero. We provide here a much simpler proof of this fact based in \cite{rhrhtu2009}. Also, since $C^{1+\alpha}$ symplectomorphisms are dense in $\diff^1_\omega(M)$, the authors state the result immediately follows for generic $f\in\diff^1_\omega(M)$. However, since the measure of the hyperbolic sets varies only upper-semicontinuous, and we are not necessarily dealing with isolated hyperbolic sets, the conclusion is not so apparent. We also provide here a complete proof.
\begin{proof}
First of all, let us prove any $C^{1+\alpha}$ conservative diffeomorphisms, having a hyperbolic set with positive measure is Anosov. Let $f\in\diff^1_m(M)$ and $\Lambda$ a hyperbolic set with positive measure. It follows from \cite[Lemma 2.3]{rhrhtu2009} applied to the characteristic function of $\Lambda$, that $\Lambda$ contains all its stable and unstable leaves. Therefore, it is both a hyperbolic attractor and repeller. Hence, $\Lambda=M$.\par
Now, let ${\mathcal H}_{l,n}$ the set of diffeomorphisms having a hyperbolic set $\Lambda$ of measure greater or equal than $\frac1n$ such that $\|Df^l|_{E^s_\Lambda}\|\leq \frac12$. Since Lebesgue measure is upper-semicontinuous with respect to Hausdorff distance, it is easy to see that the sets ${\mathcal H}_{l,n}$ are closed. Indeed, if $f_k$ has a hyperbolic set $\Lambda_k$ with measure greater or equal than $\frac1n$, over which $\|Df_k^l|_{E^s_\Lambda}\|\leq \frac12$, it is obvious that the limit $f$ of the sequence $f_k$ will belong to ${\mathcal H}_{l,n}$. \par
Now, due to \cite{avila2008}, and the result in the first paragraph of this proof, we have that $l$-Anosov diffeomorphisms are dense in the interior of each ${\mathcal H}_{l,n}$. Hence the closure of the interior of each ${\mathcal H}_{l,n}$ consists of Anosov diffeomorphisms. Now, generically in $\diff^1_m(M)$ or $\diff^1_\omega(M)$, if $f$ has a hyperbolic set with positive measure, then it belongs to the closure of the interior of some ${\mathcal H}_{l,n}$, and hence it is Anosov. This proves the claim.
\end{proof}
For 3 dimensional manifolds, the possibilities are more. But after \cite{avila_bochi2009} and \cite{bochi_viana2005} we can group them into three situations. Namely, generically in $\diff^3_m(M)$, we have one of the following situations:
\begin{enumerate}
\item all Lyapunov exponents vanish almost everywhere, 
\item $\nuh(f)\circeq M$, $f$ is ergodic and the zipped Oseledets splitting extends to a global dominated splitting, or
\item there is a partially hyperbolic set $\Lambda$ with $m(\Lambda)>0$
\end{enumerate}
A partially hyperbolic set is a set over which there is a partially hyperbolic splitting, see Definition \ref{definition partially hyperbolic}. Now, in \cite{jrhertz2010} the work of the author consists in showing an analogous to Proposition \ref{proposition.hyperbolic.sets} for partially hyperbolic sets, namely:
\begin{theorem}\cite{jrhertz2010}\label{teo.jana.partially.hyperbolic.sets}
Given $r\in[1,\infty]$, for a generic conservative diffeomorphism $f$ in $\diff^r_m(M^3)$, either:
\begin{enumerate}
\item $f$ is partially hyperbolic, or
\item all partially hyperbolic sets have measure zero.
\end{enumerate}
\end{theorem}
The proof of Theorem \ref{teo.jana.partially.hyperbolic.sets} involves delicate arguments that are specific for 3-dimensional manifolds. Indeed, the geometry of the intersticial zones of the complement of a lamination is involved, as well as the fact that the accessibility classes are manifolds, something that, so far, it is only known for partially hyperbolic dynamics with one-dimensional center \cite{rhrhu2008}.\newline\par
The general case involves many possible difficulties. For instance, in higher dimensions, one can have (stable) non-uniform hyperbolic examples that are not hyperbolic nor partially hyperbolic \cite{bonatti_viana2000}. These examples can be even stably ergodic \cite{tahzibi2004}.\par
Let us consider some special cases that could be interesting approaches to Conjecture \ref{conjecture.avila.bochi}, and are still unsolved. As Proposition \ref{proposition.jana.ergodicidad} below shows, under the assumption of ergodicity, the conjecture easily follows:
\begin{proposition}\label{proposition.jana.ergodicidad}
For a generic ergodic diffeomorphisms $f$ in $\diff^1_m(M)$, either:
\begin{enumerate}
\item all Lyapunov exponents vanish almost everywhere, or
\item the Oseledets splitting extends to a global dominated splitting and $\nuh(f)\circeq M$.
\end{enumerate}
\end{proposition}
\begin{proof}
If there is a set $Z$ over which all Lyapunov exponents vanish, then by ergodicity, either $m(Z)=1$ or $m(Z)=0$. In the first case, the conjecture is proved, so let us assume $m(Z)=0$. Then \cite{bochi_viana2005} implies that the Oseledets splitting over almost every orbit is dominated. Since almost every orbit is dense, the finest Oseledets splitting extends to a globally dominated splitting. Note that, since $f$ is ergodic, we have that $m$-almost every $x$: 
$$k\hat\lambda_i(x)=\int_M\log J^if(x)\,dm$$
where $J^if(x)$ is the Jacobian corresponding to $f|E_i$, and $k$ is the multiplicity of $\hat\lambda_i(x)$. Also, the amount in the right term is continuous with respect to $f$. Hence, if $\nuh(f)\circeq M$, then for a generic ergodic diffeomorphism we have the claim. Otherwise, there is one (and only one, due to domination) $\hat\lambda_i(x)=0$. Now Baraviera and Bonatti \cite{baraviera_bonatti2003} state that a little perturbation makes $\int_M\log J^if(x)\,dm\ne 0$. Proposition \ref{proposition.jana.ergodicidad} is then proved. 
\end{proof}
It makes sense to talk about generic ergodic diffeomorphism, since ergodic diffeomorphisms are a $G_\delta$ set in $\diff^1_m(M)$, see for instance \cite{oxtoby_ulam1941}. However, we do not know if this $G_\delta$ set is dense. \par
As we have mentioned in the introduction, it is not even known if {\em weak ergodicity} (Definition \ref{def.weak.ergodicity}) is generic among $\diff^1_m(M)$, though it seems easier to satisfy this hypothesis. Here the technique in \cite{baraviera_bonatti2003} cannot be immediately applied. In any case, we have the following:
\begin{proposition}\label{proposition.weak.ergodic}
For a generic weakly ergodic diffeomorphism $f$ in $\diff^1_m(M)$, one of the following holds. Either:
\begin{enumerate}
\item all Lyapunov exponents vanish almost everywhere, or
\item the finest Oseledets splitting extends to a global dominated splitting.
\end{enumerate}
\end{proposition}
There is a specific case where the result is still unsolved and satisfies hypotheses above: Indeed, as we stated in Proposition \ref{proposition.weak.ergodicity.generic}, weak ergodicity is generic among partially hyperbolic diffeomorphisms. This is a consequence of the fact proved by Brin that accessibility implies weak ergodicity for conservative smooth diffeomorphisms \cite{brin1975}, and that accessibility contains an open and dense set of conservative diffeomorphisms \cite{dolgopyat_wilkinson2003}, as it has been shown by Dolgopyat and Wilkinson. However, even in this simpler case, the following question is still unsolved:
\begin{question}\label{question.partially.hyperbolic}
Are non-uniform hyperbolicity and ergodicity generic among partially hyperbolic diffeomorphisms in $\diff^1_m(M)$?
\end{question}
Note that solving Conjecture \ref{conjecture.pugh.shub} below would give a positive answer to Question \ref{question.partially.hyperbolic}:
\begin{question}[Pugh-Shub \cite{pugh_shub1996,pugh_shub2000}]\label{conjecture.pugh.shub} Stable ergodicity is dense in $\diff^r_m(M)$.
\end{question}
A diffeomorphism $f\in\diff^1_m(M)$ is {\em stably ergodic} if all $C^1$-nearby diffeomorphisms $g\in\diff^{1+\alpha}$ are ergodic. This conjecture was proven true for $r=\infty$ in the case that the center bundle is one-dimensional, by F. Rodriguez Hertz, J. Rodriguez Hertz and Ures in \cite{rhrhu2008}. It is also known for $r=1$ and two-dimensional center bundle, as it has been shown by F. Rodriguez Hertz, J. Rodriguez Hertz, Tahzibi and Ures \cite{rhrhtu2009}. \par
In \cite{rhrhtu2009} we arrive to a situation as described in Proposition \ref{proposition.weak.ergodic}. Now, a perturbation like in \cite{baraviera_bonatti2003} not necessarily makes $m(\nuh(f))>0$, since the diffeomorphism is not known to be ergodic. However, using \cite{rhrhtu2010}, one can ``blend" the zone where there is at least one positive center Lyapunov exponent, with the zone where there is at least one negative center Lyapunov exponent. We believe that a similar procedure could be used to prove Conjecture \ref{conjecture.avila.bochi} under the hypothesis of generic weak ergodicity, what would solve, for instance, Question \ref{question.partially.hyperbolic}.

\end{document}